\documentclass[11pt]{article}

\usepackage[margin=1in]{geometry}
\usepackage{amsmath,amssymb,amsthm,mathrsfs,mathtools}
\usepackage[T1]{fontenc}
\usepackage{lmodern}
\usepackage{hyperref}
\usepackage{enumitem}
\usepackage{microtype}

\hypersetup{colorlinks=true, linkcolor=blue, citecolor=blue, urlcolor=blue}
\usepackage[
backend=biber,
style=alphabetic,
giveninits=true,
url = false,
doi = false,
isbn =false,
maxbibnames=99,
maxalphanames=99,
]{biblatex}
\renewbibmacro{in:}{}
\DeclareFieldFormat{title}{#1} 
\DeclareFieldFormat[article]{title}{#1}
\DeclareFieldFormat[inbook]{title}{#1}
\DeclareFieldFormat[incollection]{title}{#1}
\numberwithin{equation}{section}

\newtheorem{theorem}{Theorem}[section]
\newtheorem{assumption}[theorem]{Assumption}
\newtheorem{lemma}[theorem]{Lemma}
\newtheorem{proposition}[theorem]{Proposition}

\newtheorem{definition}[theorem]{Definition}

\newcommand{\dif}{d}

\newcommand{\R}{\mathbb{R}}
\newcommand{\E}{\mathbb{E}}
\newcommand{\Pbb}{\mathbb{P}}
\newcommand{\Var}{\operatorname{Var}}

\newcommand{\Tr}{\operatorname{Tr}}
\newcommand{\1}{\mathbf{1}}
\newcommand{\scdens}{\rho_{\mathrm{sc}}}
\newcommand{\ip}[2]{\langle #1,#2\rangle}
\newcommand{\cO}{\mathcal{O}}
\DeclareMathOperator{\supp}{supp}
\DeclareMathOperator{\sgn}{sgn}
\DeclareMathOperator{\pv}{p.v.}

\date{\today}

\begin{document}

\begin{table}
\centering

\begin{tabular}{c}

\multicolumn{1}{c}{\parbox{12cm}{\begin{center}\Large{\bf Sobolev convergence of log-determinants for smooth Wigner matrices}\end{center}}}\\
\\
\end{tabular}
\begin{tabular}{ c c c  }
Giorgio Cipolloni
& \phantom{blah} & %\qquad
Patrick Lopatto
 \\
 & & \\  
 \small{University of Rome Tor Vergata} & & \small{University of North Carolina at Chapel Hill} \\
 \small{Department of Mathematics} & & \small{Department of Statistics and Operations Research} \\
 \small{\texttt{cipolloni@axp.mat.uniroma2.it}} & & \small{\texttt{lopatto@unc.edu}} \\
  & & \\
\end{tabular}
\\
\begin{tabular}{c}
\multicolumn{1}{c}{\today}\\
\\
\end{tabular}

\begin{tabular}{p{15 cm}}
\small{{\bf Abstract:} We show that the fields emerging from the log-determinant and the eigenvalue counting function of smooth Wigner matrices converge in law to centered Gaussian, logarithmically correlated, random elements in every
negative Sobolev space $H^{-s}$.
}
\end{tabular}
\end{table}

\section{Introduction}

We consider Wigner matrices, which are $N\times N$ symmetric matrices $H_N=H_N^\top$ with independent 
upper-triangular  entries normalized so that $\E (H_N)_{ij}=0$, $\E (H_N)_{ij}^2=N^{-1}$ for $i\neq j$, and $\E (H_N)_{ii}^2=2N^{-1}$ for all $i$. It is well known  that the (random) empirical measure of the eigenvalues $\lambda_1(H_N),\dots,\lambda_N(H_N)$ of $H_N$ converges to Wigner's semicircular law, $\mu_{\mathrm{sc}}(\dif x)=\frac{1}{2\pi}\sqrt{(4-x^2)_+}\,\dif x$ \cite{wigner1958distribution}. It is then natural to study the fluctuations around this deterministic limit. From the breakthrough work \cite{Johansson04} (see also \cite{LP09}), it is known that the fluctuations around $\mu_{\mathrm{sc}}$ are Gaussian and appear on a surprisingly small scale, $N^{-1}$, instead of the usual $N^{-1/2}$ scale typical of central limit theorems. These results show that
\begin{equation}
\label{eq:CLT}
\sum_i f(\lambda_i)-\E \sum_i f(\lambda_i)\Longrightarrow \mathcal{N}(0,V_f)
\end{equation}
for sufficiently regular test functions $f$, where $\mathcal{N}(0,V_f)$ is a Gaussian random variable and the variance $V_f>0$ is explicit. Starting from \cite{Johansson04, LP09}, there has been great effort directed toward showing that the convergence in \eqref{eq:CLT} holds for more and more general classes of test functions \cite{bao2016clt, Sh11, SW13} (see also previous works \cite{girko2012theory, sinai1998central} for polynomial test functions). This series of works culminated in a recent result of Landon and Sosoe \cite{LS22}, where they proved that \eqref{eq:CLT} holds for the (almost) optimal class of test functions $f\in H^{1/2+\varepsilon}(I)$, for any bulk interval $I \Subset (-2,2)$.

It is natural to ask whether the results above apply to the  observable $\log|\det(H_N-E)|$, which has recently attracted lots of attention (see e.g.\ \cite{augeri2025maximum, bourgade2025optimal, lambert2019law}) in connection with the Fyodorov--Hiary--Keating conjecture \cite{fyodorov2012freezing} (see also \cite{fyodorov2016distribution}). The log-determinant $\log|\det(H_N-E)|$ can be realized as a linear eigenvalue statistic for the choice of test function $f_E(\cdot)=\log |\cdot-E|$, which is however not in $H^{1/2}$. Hence, the above results do not immediately apply. Inspired by analogous results for Haar unitary matrices \cite{hughes2001characteristic},  Fyodorov, Khoruzhenko, and Simm considered the centered log-determinant field
\[
L_N(E)=\log |\det(H_N-E)|-\E \log |\det(H_N-E)|,
\]
when $H_N$ is a GUE matrix  \cite{fyodorov2016fractional}.\footnote{We say that a matrix $H_N$ belongs to the Gaussian Unitary Ensemble (GUE) if the strictly upper-triangular entries of $\sqrt{N}H_N$ are complex standard Gaussians and the diagonal ones are standard real Gaussians.} In this context,  $L_N(E)$ is naturally viewed as a random distribution on $I$. Relying on the integrable structure of the joint distribution of GUE eigenvalues, they showed that $L_N(E)$ converges in the Sobolev space $H^{-s}$, for any $s>1/2$, to the random Chebyshev--Fourier series
\[
\sum_{n\ge 1}\frac{a_n}{\sqrt{n}}T_n(E),
\]
where $a_n$ are i.i.d.\ standard real Gaussians and $T_n(E)$ are the Chebyshev polynomials. Note that the limiting field is logarithmically correlated, 
as $\sum_n\frac{2}{n}T_n(E)T_n(E')=-\log(2|E-E'|)$ (as shown below in Section~\ref{s:logfieldcovariance}).

We  study the field $L_N$ for  Wigner matrices  with smooth entry distributions and prove a similar 
convergence result for any $s>0$.  
In addition, we also consider a similar question for the centered counting field
\[
\mathfrak{n}_N(E)=\#\{j:\lambda_j(H_N)\le E\}-\E \#\{j:\lambda_j(H_N)\le E\},
\]
which is again viewed as a random distribution on $I$. 
More precisely, we show that, after testing against smooth test functions on compact intervals $I \Subset (-2,2)$, both of these fields converge in law to centered Gaussian log-correlated random elements in every negative Sobolev space $H^{-s}(I)$, with explicit, logarithmically correlated covariances.

Our proof relies on two main ingredients. First, we prove convergence of the random fields tested against any element of a dense set of functions by rewriting the fields as centered linear statistics and invoking the bulk $H^{1/2+\varepsilon}(I)$ central limit theorem of Landon and Sosoe \cite{LS22}, together with the classical smooth linear-statistics CLT of Lytova--Pastur \cite{LP09} and Shcherbina \cite{Sh11}. Second, we prove a uniform bound on the Sobolev norms $\E\|L_N\|_{H^{-r}(I)}^2$ and $\E\|\mathfrak{n}_N\|_{H^{-r}(I)}^2$ for every $r>0$.  Combining these bounds with the compact embedding $H^{-r}(I)\hookrightarrow H^{-s}(I)$, we obtain tightness in every $H^{-s}(I)$, and conclude Sobolev convergence.

\subsection*{Acknowledgements} The research of G. C. is supported by the Italian Ministry of University and Research (MUR) - Fondo Italiano per la Scienza (FIS3) - 2024 Call, project UBLOCO, CUP F53C25000940001, and also partially supported by the MUR Excellence Department Project MatMod@TOV awarded to the Department of Mathematics, University of Rome Tor Vergata, CUP E83C23000330006. Additionally, G.C. thanks INdAM (Istituto Nazionale di Alta Matematica ``Francesco Severi'') and the group GNFM.
P. L. was partially supported by NSF grant DMS-2450004. 
This manuscript was written with the assistance of large language models, including aid in drafting, revising, and proofreading. We also thank Peter Forrester for bringing to our attention \cite[Exercises 1.4 q.4]{forrester2010log}.

\subsection*{Notations and conventions} 
For functions $f\colon \R \rightarrow \R$, we use the Fourier transform convention
\[
    \widehat f(\xi)
    =
    \int_{\mathbb R} e^{-i y \xi} f(y)\,\dif y.
\]
For all $s>0$ we define the Sobolev space
\[
H^s(\R)=\left\{u\in L^2(\R): \lVert u\rVert_{H^s(\R)} <\infty\right\},\qquad 
\|u\|_{H^s(\mathbb R)}^2
=
\int_{\mathbb R}(1+|\xi|^2)^s |\widehat u(\xi)|^2\,d\xi. 
\]

Fix a compact interval $I=[a,b]\Subset (-2,2)$, 
and write $L=b-a$. Let $(e_k,\mu_k)_{k\ge1}$ be the Dirichlet eigenbasis of $-\partial_x^2$ on $I$:
\begin{equation}\label{eq:dirichlet-basis}
e_k(E)=\sqrt{\frac{2}{L}} \sin\left(\frac{\pi k(E-a)}{L}\right),
\qquad
\mu_k=\left(\frac{\pi k}{L}\right)^2.
\end{equation}
For $s>0$ and for $u\in L^2(I)$, we set
\[
\|u\|_{H^{-s}(I)}^2
=
\sum_{k\ge1}(1+\mu_k)^{-s} \ip{u}{e_k}_{L^2(I)}^2,
\]
and define $H^{-s}(I)$ as the Hilbert space obtained as the completion of $L^2(I)$ with respect to this norm. 
We write $H_0^s(I)$  for the dual of $H^{-s}(I)$, which can be defined explicitly as 
\[
H_0^s(I)
=
\left\{
h=\sum_{k\ge1}h_k e_k:
\sum_{k\ge1}(1+\mu_k)^s |h_k|^2<\infty
\right\}.
\]  
Throughout, \(C_c^\infty(I)\) means \(C_c^\infty((a,b))\), viewed as
functions on \(I=[a,b]\) by restriction.

\section{Main Results}\label{s:main}

We start with the definition of smooth probability density. 
\begin{definition}[Wigner-smooth densities]\label{def:wigner-smooth-density}
A probability density $h$ on $\mathbb R$ is called Wigner-smooth if $h$ is strictly positive, $h\in C^\infty(\mathbb R)$,  there exists $\delta>0$ such that 
\begin{equation}\label{eq:wigner-smooth-tail}
\int_{\R} e^{\delta |x|^2} h(x)\,dx < \infty,
\end{equation}
and for every integer $k\ge0$ there exists $C_k>0$ such that
\begin{equation}\label{eq:wigner-smooth-derivatives}
\left| \frac{d^k}{dx^k}\log h(x)\right| \le C_k (1+|x|)^{C_k}, \qquad x\in\R.
\end{equation}
\end{definition}

We consider $N\times N$ real symmetric\footnote{We assume that the entries of $H_N$ are real for notational simplicity. Our arguments and results hold in the complex Hermitian  case as well, after making appropriate changes to certain constant factors.}  random matrices $H_N$ satisfying the following assumption. 
\begin{assumption}
\label{assumption:hn}
The matrix entries $(H_{ij})_{1\le i,j\le N}$ of $H_N$ are independent (up to  symmetry) and satisfy $H_{ij}\stackrel{\dif}{=}N^{-1/2}\chi_{\mathrm{od}}$ for $i<j$, and $H_{ii}\stackrel{\dif}{=}N^{-1/2}\chi_{\dif}$ for all $i$, with $\chi_{\dif}, \chi_{\mathrm{od}}$ being $N$-independent smooth real random variables (in the sense of Definition~\ref{def:wigner-smooth-density}) with $\E\chi_{\mathrm{od}}=\E\chi_{\dif}=0$ and $\E|\chi_{\mathrm{od}}|^2=1$, $\E|\chi_{\dif}|^2=2$. Additionally, let $s_k(X)$ denote the $k$th cumulant of a real random variable $X$. We also assume\footnote{We stress that inspecting our proof it is clear that the assumptions $\E|\chi_{\dif}|^2=2$ and \eqref{eq:wigner-cumulant} can easily be weakened, however we refrain from doing so for the sake of clarity, since for certain steps of the proof 
we refer to \cite{LS22} which uses these additional assumptions for convenience (see \cite[Definition~1.1]{LS22} and the following comments).}
\begin{equation}\label{eq:wigner-cumulant}
s_k(\chi_{\dif})=2^{k-1}s_k(\chi_{\mathrm{od}}), \qquad 3\le k\le4,
\end{equation}
and that for every integer $k\ge1$ there exists $C_k>0$ such that
\begin{equation}\label{eq:wigner-moments}
\E \big[|\chi_{\dif}|^k+|\chi_{\mathrm{od}}|^k\big] \le C_k.
\end{equation}

\end{assumption}
Using the notation of the previous assumption, we abbreviate $s_4 = s_4(\chi_{\mathrm{od}})$.

The asymptotic eigenvalue density profile of such $H_N$ is described by the semicircular law, which has density
\[
\scdens(x)=\frac{1}{2\pi}\sqrt{(4-x^2)_+}.
\]

We now want to study the convergence of the $\log$-determinant and of the eigenvalue counting function as random processes. For this purpose we define the covariance kernels:
\begin{definition}[Covariance kernels]
\emph{ For any compact interval $I\Subset (-2,2)$, and for $E,E'\in I$, we define the kernels}
\begin{equation}\label{eq:Klog}
K_{\mathrm{log}}(E,E')
=
-\log|E-E'|+\frac{s_4}{8}(2-E^2)(2-E'^2),
\end{equation}
\emph{and}
\begin{align}
K_{\mathrm{cnt}}(E,E')
&=
-\frac{1}{\pi^2}\log|E-E'|
\notag+\frac{1}{\pi^2}
\log\left(
\frac{4-EE'+\sqrt{4-E^2}\sqrt{4-E'^2}}{2}
\right)
\notag\\
&\quad
+\frac{s_4}{8\pi^2} EE' \sqrt{4-E^2}\sqrt{4-E'^2}.
\label{eq:Kcnt}
\end{align}
\end{definition}

For $\phi,\psi\in C_c^\infty(I)$ we also define the bilinear forms
\begin{equation}\label{eq:Vlog}
V_{\mathrm{log}}(\phi,\psi)
=
\iint_{I\times I}\phi(E)K_{\mathrm{log}}(E,E')\psi(E')\,dE\,dE'
\end{equation}
and $V_{\mathrm{cnt}}(\phi,\psi)$ is defined by replacing $K_{\mathrm{log}}$ with $K_{\mathrm{cnt}}$. 
These bilinear forms extend continuously to $H_0^s(I)$ for every $s>0$ (see Proposition \ref{prop:covariance-continuity} below); we use this fact implicitly in the theorem statements below.

We are now ready to state our main theorems. For $E\in I$, define the (centered) $\log$-determinant
\[
L_N(E)=\log|\det(H_N-E)|-\E \log|\det(H_N-E)|
\]
and, for $\phi\in C_c^\infty(I)$, let
\begin{equation}\label{eq:log-pairing}
\ip{X_N^{\mathrm{log}}}{\phi}
=
\int_I \phi(E)L_N(E)\,dE.
\end{equation}

\begin{theorem}\label{thm:log-main}
Fix any $s>0$. Let $H_N$ be a smooth real symmetric Wigner matrix satisfying Assumption~\ref{assumption:hn}. Then 
\[
X_N^{\mathrm{log}} \Rightarrow X^{\mathrm{log}}
\]
in law in $H^{-s}(I)$, where $X^{\mathrm{log}}$ is the centered Gaussian random element of $H^{-s}(I)$ uniquely determined by
\[
\E\bigl[\ip{X^{\mathrm{log}}}{h}\ip{X^{\mathrm{log}}}{g}\bigr]
=
V_{\mathrm{log}}(h,g),
\qquad
h,g\in H_0^s(I).
\]
\end{theorem}
The proof of Theorem~\ref{thm:log-main} is given in Section~\ref{s:logfieldproof}. 
Previously the convergence in Theorem~\ref{thm:log-main} was known only in the Gaussian case  and for $s>1/2$ \cite[Theorem 2.1]{fyodorov2016fractional}. Our result holds for a fairly general class of real symmetric matrices, delineated in Assumption~\ref{assumption:hn}, and the convergence holds in a stronger sense (even in the Gaussian case), as Theorem~\ref{thm:log-main} holds for any $s>0$.

In \cite{fyodorov2016fractional}, the limiting field $X^{\mathrm{log}}$ was
identified, in the complex Hermitian Gaussian case and in the normalization where the spectrum
is supported on $[-1,1]$, as the random Chebyshev--Fourier series
\[
    \sum_{n\ge 1}\frac{a_n}{\sqrt n}T_n(x),
\]
where $a_n$ are independent standard real Gaussians and $T_n$ are the
Chebyshev polynomials (of the first kind). We remark that this series must be understood as a random distribution, as it does not define a true random function.  

For Wigner matrices, in the present normalization the spectrum is supported
on $[-2,2]$, so the corresponding Chebyshev modes are $
    \widetilde T_n(E):=T_n(E/2)$. 
Since
\[
    -\log |E-E'|
    =
    2\sum_{n\ge 1}\frac{1}{n}
        \widetilde T_n(E)\widetilde T_n(E'),
\]
as is well known (and shown below in Section~\ref{s:logfieldcovariance}), the Gaussian real-symmetric
field with covariance kernel $-\log |E-E'|$ is
\[
    \sqrt{2}\sum_{n\ge 1}\frac{a_n}{\sqrt n}\widetilde T_n(E).
\]
The factor of $\sqrt{2}$ difference with  \cite{fyodorov2016fractional} comes from the change in symmetry class. 
Jensen's inequality guarantees that $1+s_4/2\ge 0$, and one still has
the explicit representation
\[
    X^{\log}(E)
    =
    \sum_{n\ne 2}\sqrt{\frac{2}{n}}\,a_n\widetilde T_n(E)
    +
    \sqrt{1+\frac{s_4}{2}}\,a_2\widetilde T_2(E),
\]
where we again note that this series is understood as a random distribution. 
This absorbs the $s_4$-correction into the variance of the second
Chebyshev mode. It follows from the definition of $K_{\mathrm{log}}$ and the definition of $\widetilde T_2$.

Additionally, the limiting field acts on test functions $f$ as
\[
    X^{\log}[f]
    =
    \sum_{n\ne 2}\sqrt{\frac{2}{n}}\,a_n d_n(f)
    +
    \sqrt{1+\frac{s_4}{2}}\,a_2 d_2(f),
    \qquad
    d_n(f):=\int_I f(E)\widetilde T_n(E)\,\dif E .
\]
Consequently,
\[
    \E X^{\log}[f]X^{\log}[g]
    =
    \sum_{n\ge1}\frac{2}{n}d_n(f)d_n(g)
    +
    \frac{s_4}{2}d_2(f)d_2(g)
    =
    V_{\mathrm{log}}(f,g).
\]

To compare this with the usual CLT for smooth linear eigenvalue statistics,
let
\[
    F(\lambda):=\int_I f(E)\log|\lambda-E|\,\dif E
\]
and define the arcsine Chebyshev coefficients
\[
    \alpha_n(F):=
    \frac{1}{\pi}\int_{-2}^2
        \frac{F(\lambda)\widetilde T_n(\lambda)}
             {\sqrt{4-\lambda^2}}\,\dif\lambda .
\]
The Chebyshev identity
\[
    \frac{1}{\pi}\int_{-2}^2
        \frac{\log|\lambda-E|\,\widetilde T_n(\lambda)}
             {\sqrt{4-\lambda^2}}\,\dif\lambda
    =
    -\frac{1}{n}\widetilde T_n(E),
    \qquad n\ge1
\]
implies $d_n(f)=-n\alpha_n(F)$. 
Therefore
\[
    V_{\mathrm{log}}(f,f)
    =
    \sum_{n\ge1}2n\,\alpha_n(F)^2
    +
    2s_4\,\alpha_2(F)^2.
\]
This is precisely the real-symmetric 
variance formula for smooth linear eigenvalue statistics, recovering
\cite[Eq.\ (1.5)]{bao2016clt}.

We now state an analogous result for the eigenvalue counting function. For $E\in I$, define
\[
\mathfrak{n}_N(E)=\#\{j:\lambda_j(H_N)\le E\}-\E \#\{j:\lambda_j(H_N)\le E\},
\]
and, for every $\phi\in C_c^\infty(I)$,
\begin{equation}\label{eq:count-pairing}
\ip{X_N^{\mathrm{cnt}}}{\phi}
=
\int_I \phi(E)\mathfrak{n}_N(E)\,dE.
\end{equation}

\begin{theorem}\label{thm:count-main}
Fix any $s>0$. Let $H_N$ be a smooth real symmetric Wigner matrix satisfying Assumption~\ref{assumption:hn}. Then 
\[
X_N^{\mathrm{cnt}} \Rightarrow X^{\mathrm{cnt}}
\]
in law in $H^{-s}(I)$, where $X^{\mathrm{cnt}}$ is the centered Gaussian random element of $H^{-s}(I)$ uniquely determined by
\[
\E\bigl[\ip{X^{\mathrm{cnt}}}{h}\ip{X^{\mathrm{cnt}}}{g}\bigr]
=
V_{\mathrm{cnt}}(h,g),
\qquad
h,g\in H_0^s(I).
\]
\end{theorem}
The proof of Theorem~\ref{thm:count-main} is given in Section~\ref{s:countingfieldproof}. 
\section{Preliminary Results}

In this section, we recall some standard results that will be needed for our work. Section \ref{s:functional} collects results from functional analysis, and Section \ref{s:linearstats} collects results on linear statistics of random matrix eigenvalues. 

\subsection{Functional Analysis Background}\label{s:functional}

\begin{proposition}\label{prop:compact-embedding}
If $0<r<s$, then the embedding $
H^{-r}(I)\hookrightarrow H^{-s}(I)$ 
is compact.
\end{proposition}
\begin{proof}
Let $(u_n)_{n=1}^\infty$ be any bounded sequence in $H^{-r}(I)$, and recall $e_k, \mu_k$ from \eqref{eq:dirichlet-basis}. Writing $u_{n,k}=\ip{u_n}{e_k}$, we have
\[
\sup_n \sum_{k\ge1}(1+\mu_k)^{-r}|u_{n,k}|^2<\infty.
\]
Fix $M\ge1$ and let $P_M$ be the orthogonal projection onto $\operatorname{span}\{e_1,\dots,e_M\}$. Then,
\[
\|u_n-P_Mu_n\|_{H^{-s}(I)}^2
=
\sum_{k>M}(1+\mu_k)^{-s}|u_{n,k}|^2
\le
(1+\mu_M)^{-(s-r)}\|u_n\|_{H^{-r}(I)}^2,
\]
where the last inequality follows from the fact that the $\mu_k$ are increasing in $k$. 
The right-hand side tends to $0$ uniformly in $n$ as $M\to\infty$. Since $P_M$ has finite rank, $(P_Mu_n)_{n=1}^\infty$ has a convergent subsequence in $H^{-s}(I)$ for every fixed $M$. Then a standard diagonal argument shows that  $(u_n)_{n=1}^\infty$ has a convergent subsequence in $H^{-s}(I)$, which implies the conclusion.
\end{proof}

\begin{proposition}\label{prop:tightness-from-stronger}
Let $0<r<s$ and let $(Y_N)$ be a sequence of $H^{-s}(I)$-valued random variables such that
\[
\sup_N \E \|Y_N\|_{H^{-r}(I)}^2 < \infty.
\]
Then $(Y_N)$ is tight in $H^{-s}(I)$.
\end{proposition}

\begin{proof}
For $R>0$, let
\[
B_R^{(-r)}=\{u\in H^{-r}(I):\|u\|_{H^{-r}(I)}\le R\}.
\]
By Proposition \ref{prop:compact-embedding}, $B_R^{(-r)}$ is compact when viewed as a subset of $H^{-s}(I)$. Markov's inequality gives
\[
\sup_N \Pbb\big (Y_N\notin B_R^{(-r)}\big )
\le
R^{-2}\sup_N \E \|Y_N\|_{H^{-r}(I)}^2.
\]
Taking $R\to\infty$ yields the conclusion. 
\end{proof}

\begin{proposition}\label{prop:dense-characteristic}
Let $\mu$ and $\nu$ be Borel probability measures on $H^{-s}(I)$. If their characteristic functionals agree on a dense subset of $H_0^s(I)$, then $\mu=\nu$.
\end{proposition}

\begin{proof}
Let $\chi_\mu$ and $\chi_\nu$ be the characteristic functionals of $\mu$ and $\nu$:
\[
\chi_\mu(h)=\int_{H^{-s}(I)} e^{i\ip{u}{h}}\,d\mu(u),
\qquad
\chi_\nu(h)=\int_{H^{-s}(I)} e^{i\ip{u}{h}}\,d\nu(u),
\qquad h\in H_0^s(I).
\]
If $h_n\to h$ in $H_0^s(I)$, then for every $u\in H^{-s}(I)$, we have 
$
e^{i\ip{u}{h_n}}\to e^{i\ip{u}{h}}$, 
and the integrand is bounded in absolute value by $1$. The dominated convergence theorem therefore shows that both $\chi_\mu$ and $\chi_\nu$ are continuous on $H_0^s(I)$. By assumption the two functionals agree on a dense subset of $H_0^s(I)$, hence they agree everywhere by continuity. The Hilbert-space version of Lévy's uniqueness theorem then implies $\mu=\nu$; see, for example, \cite[(2.4)]{da2014stochastic}.
\end{proof}

\begin{proposition}\label{prop:convergence-criterion}
Fix $s>0$. Let $(Y_N)$ be a sequence of $H^{-s}(I)$-valued random variables. Assume that $(Y_N)$ is tight and that there exists a countable dense subset $\mathcal D\subset H_0^s(I)$ such that for every $\phi \in \mathcal D$, the sequence $\ip{Y_N}{\phi}$ 
converges in law as $N\to\infty$. Then $(Y_N)$ converges in law in $H^{-s}(I)$.
\end{proposition}

\begin{proof}
By tightness, every subsequence of $(Y_N)$ admits a further weakly convergent subsequence in $H^{-s}(I)$. Let $Y$ and $Y'$ be two subsequential limits. Then for every $\phi \in \mathcal D$, the laws of $\ip{Y}{\phi}$ and $\ip{Y'}{\phi}$ 
coincide by assumption. Proposition \ref{prop:dense-characteristic} therefore implies that $Y$ and $Y'$ have the same law. Hence all subsequential limits agree, and the full sequence $(Y_N)$ converges.
\end{proof}

\subsection{Linear Statistics}\label{s:linearstats}

Fix an auxiliary bulk interval $J\subset \R$ such that 
\[
I\Subset J\Subset (-2,2).
\]
Fix cutoff functions $\chi,\eta\in C_c^\infty(\R)$ such that
\begin{equation}\label{eq:cutoffs}
\chi\equiv1 \text{ on a neighborhood of } I,
\qquad
\supp \chi \subset J,
\qquad
\eta\equiv1 \text{ on } [-3,3].
\end{equation}

For a real-valued function $f$, set
\[
\mathcal N_N(f)=\Tr f(H_N)-\E \Tr f(H_N).
\]
For a function $f$ defined on a neighborhood of $[-2,2]$, define its Chebyshev coefficients by
\begin{equation}\label{eq:ck-def}
c_k(f)=\frac{1}{\pi}\int_{-\pi}^{\pi} f(2\cos\theta)\cos(k\theta)\,d\theta,
\qquad k\ge1.
\end{equation}
We then set
\begin{equation}\label{eq:Q-def}
Q(f)=\frac{1}{2}\sum_{k\ge1} k\, c_k(f)^2 + \frac{s_4}{2} c_2(f)^2
\end{equation}
and define
\begin{equation}\label{eq:V-def}
V(f,g)=\frac{1}{2}\bigl(Q(f+g)-Q(f)-Q(g)\bigr).
\end{equation}
Only the restriction of $f$ to $[-2,2]$ enters in \eqref{eq:ck-def}--\eqref{eq:V-def}.

The following theorem is a combination of \cite[Theorem~1.4 and Corollary~1.5]{LS22} and the variance estimate in \cite[(1.23)]{LS22}.
\begin{theorem}\label{thm:LS}
Let $H_N$ satisfy Assumption~\ref{assumption:hn}. Fix $\varepsilon>0$ and $\kappa>0$. Then there exists $C_{\varepsilon,\kappa}>0$ such that for every real-valued $\varphi\in H^{1/2+\varepsilon}(\R)$ supported in $(-2+\kappa,2-\kappa)$,
\begin{equation}\label{eq:LS-variance}
\Var\bigl(\mathcal N_N(\varphi)\bigr)
\le
C_{\varepsilon,\kappa}\|\varphi\|_{H^{1/2+\varepsilon}(\R)}^2.
\end{equation}
Moreover, $\mathcal N_N(\varphi)$ converges in law to a centered Gaussian random variable with variance $Q(\varphi)$.
\end{theorem}

The next theorem is an immediate consequence of 
\cite[Theorem~1]{SW13}.  We point out that the main difference compared to the previous statement is that in the following theorem the test functions are not necessarily supported in the bulk of the spectrum. 
\begin{theorem}\label{thm:SW}
Let $H_N$ satisfy Assumption~\ref{assumption:hn}. For every $\delta>0$ there exists $C_\delta>0$ such that for every real-valued $\varphi\in H^{1+\delta}(\R)$,
\begin{equation}\label{eq:SW-variance}
\Var\bigl(\mathcal N_N(\varphi)\bigr)
\le
C_\delta \|\varphi\|_{H^{1+\delta}(\R)}^2.
\end{equation}
\end{theorem}

We also need the following CLT for smooth test functions; see \cite{LP09,Sh11}. We note that for  smooth compactly supported test functions,  the limiting covariance agrees with the Chebyshev formula \eqref{eq:Q-def}; see \cite[Appendix~E]{LS22}.

\begin{proposition}\label{prop:smooth-clt}
Fix  $f\in C_c^\infty(\R)$. Then $\mathcal N_N(f)$ 
converges in law to a mean-zero Gaussian with variance $V(f,f)$.
\end{proposition}

The following proposition follows from the rigidity estimate
\cite[Theorem~2.2]{erdHos2012rigidity} for \eqref{eq:norm-tail}, and from
standard non-asymptotic bounds for the operator norm of symmetric random
matrices with  sub-Gaussian entries, e.g.\
\cite[Theorem 4.4.3]{vershynin2018high}, for \eqref{eq:norm-moment}.

\begin{proposition}\label{prop:norm-tail}
There exist constants $c,C>0$ such that
\begin{equation}\label{eq:norm-tail}
\Pbb(\|H_N\|\ge 3)\le C e^{-c(\log N)^2},
\qquad N\ge1.
\end{equation}
Moreover, for every $p\ge1$,
\begin{equation}\label{eq:norm-moment}
\sup_{N\ge1}\E \|H_N\|^p < \infty.
\end{equation}
\end{proposition}

\section{Covariance Calculations}

In this section, we establish facts about the covariance forms $V_{\mathrm{log}}$ and $V_{\mathrm{cnt}}$ that will be used in the proofs of the main theorems. In Section~\ref{s:kernelcontinuity}, we show that these forms  extend continuously from $C_c^\infty(I)$ to $H_0^s(I)$. 
In Section~\ref{s:logfieldcovariance}, we show that $V_{\mathrm{log}}$ can be written in terms of the covariance form $V$ from \eqref{eq:V-def}, and we make an analogous argument for $V_{\mathrm{cnt}}$ in Section~\ref{s:countfieldcovariance}.

\subsection{Continuity of the kernels on Sobolev spaces}\label{s:kernelcontinuity}

\begin{lemma}\label{lem:log-operator}
Let $\chi$ be the function from \eqref{eq:cutoffs}. Fix $0<\varepsilon<1/2$. There exists $C_{\varepsilon,\chi,I}>0$ such that for every $f\in H^{-1/2+\varepsilon}(I)$,
\begin{equation}\label{eq:log-operator-bound}
\left\|
\chi(\cdot)\int_I f(E)\log|\cdot-E|\,dE
\right\|_{H^{1/2+\varepsilon}(\R)}
\le
C_{\varepsilon,\chi,I}\|f\|_{H^{-1/2+\varepsilon}(I)}.
\end{equation}
\end{lemma}

\begin{proof}

Let $r =1/2-\varepsilon$. Since $r<1/2$, we have
$H^r(I)=H_0^r(I)$ with equivalent norms \cite[(3.5a)]{ern2021finite}.
Let
\[
    R_I:H^r(\mathbb R)\to H^r(I)=H_0^r(I)
\]
be the restriction map, which is bounded \cite[(2.13)]{nochetto2015pde}. Hence its adjoint defines a bounded map
\[
    R_I^*:H^{-r}(I)=(H_0^r(I))'\to (H^r(\mathbb R))'=H^{-r}(\mathbb R).
\]
For \(f\in H^{-r}(I)\), we denote \(R_I^*f\) by \(\widetilde f\).  It is precisely the zero extension of \(f\) to \(\mathbb R\), and therefore
\(f\mapsto \widetilde f\) is bounded from
\(H^{-1/2+\varepsilon}(I)\) to \(H^{-1/2+\varepsilon}(\mathbb R)\).

Choose $\zeta\in C_c^\infty(\R)$ such that $\zeta\equiv1$ on the compact set
\[
\supp\chi-I=\{x-E:x\in \supp\chi,\ E\in I\}.
\]
Set $k(y)=\zeta(y)\log|y|$. Then
\[
\chi(x)\int_I f(E)\log|x-E|\,dE
=
\chi(x)(k*\widetilde f)(x).
\]
We claim that
\begin{equation}\label{eq:k-symbol}
|\widehat{k}(\xi)|\le C\langle \xi\rangle^{-1},
\qquad
\langle \xi\rangle=(1+\xi^2)^{1/2}.
\end{equation}
Indeed, $k\in L^1(\R)$, so $\widehat{k}\in L^\infty(\R)$. Moreover,
\[
k'(y)=\zeta'(y)\log|y|+\zeta(y)\,\pv\frac{1}{y}
\]
as distributions. The first term belongs to $L^1(\R)$, and the second has bounded Fourier transform because
\[
\widehat{\pv(1/y)}(\xi)=-i\pi \sgn(\xi).
\]
Thus $\widehat{k'}\in L^\infty(\R)$, and since $\widehat{k'}(\xi)=i\xi \widehat{k}(\xi)$, we obtain \eqref{eq:k-symbol}.

By Plancherel's theorem,
\begin{align*}
\|k*\widetilde f\|_{H^{1/2+\varepsilon}(\R)}^2
&=
\int_\R \langle \xi\rangle^{1+2\varepsilon}
|\widehat{k}(\xi)|^2
|\widehat{\widetilde f}(\xi)|^2\,d\xi
\\
&\le
C\int_\R \langle \xi\rangle^{-1+2\varepsilon}
|\widehat{\widetilde f}(\xi)|^2\,d\xi
\\
&\le
C\|\widetilde f\|_{H^{-1/2+\varepsilon}(\R)}^2.
\end{align*}
Multiplication by the fixed smooth cutoff $\chi$ is bounded on $H^{1/2+\varepsilon}(\R)$, so \eqref{eq:log-operator-bound} follows.
\end{proof}

\begin{proposition}\label{prop:covariance-continuity}
For every $s>0$, the kernel formulas defining $V_{\mathrm{log}}$ and
$V_{\mathrm{cnt}}$ extend
uniquely from $\operatorname{span}\{e_k:k\ge1\}$ to continuous bilinear forms on $H_0^s(I)$. These extensions agree
with \eqref{eq:Vlog} and its counting analogue on $C_c^\infty(I)$.
\end{proposition}

\begin{proof}
Let $\mathcal S=\operatorname{span}\{e_k:k\ge1\}$, which is dense in
$H_0^s(I)$ by definition. We prove the required bound on $\mathcal S$.
Let $0<\varepsilon<1/2$. Since $\chi\equiv1$ on a neighborhood of $I$, for
$\phi,\psi\in\mathcal S$ we have
\[
\iint_{I\times I}\phi(E)\log|E-E'|\psi(E')\,dE\,dE'
=
\int_I \phi(E)\left(\chi(E)\int_I \psi(E')\log|E-E'|\,dE'\right)dE.
\]
Hence, by Cauchy--Schwarz and Lemma \ref{lem:log-operator}, and the fact that the $H^{1/2+\varepsilon}(\R)$ norm dominates the $L^2(\R)$ norm,
\begin{align*}
\left|
\iint_{I\times I}\phi(E)\log|E-E'|\psi(E')\,dE\,dE'
\right|
&\le
\|\phi\|_{L^2(I)}
\left\|
\chi\int_I \psi(E')\log|\cdot-E'|\,dE'
\right\|_{L^2(\R)}
\\
&\le
C
\|\phi\|_{L^2(I)}
\|\psi\|_{H^{-1/2+\varepsilon}(I)}.
\end{align*}
Since $s>0$ and $-1/2+\varepsilon < s$, the Sobolev embeddings
\[
H_0^s(I)\hookrightarrow L^2(I),
\qquad
H_0^s(I)\hookrightarrow H^{-1/2+\varepsilon}(I)
\]
are continuous. Therefore, combining the previous displays,
\[
\left|
\iint_{I\times I}\phi(E)\log|E-E'|\psi(E')\,dE\,dE'
\right|
\le
C_s
\|\phi\|_{H_0^s(I)}
\|\psi\|_{H_0^s(I)}.
\]
This proves continuity of the logarithmic singular part.

The remaining pieces of $K_{\mathrm{log}}$ and $K_{\mathrm{cnt}}$ are smooth on $I\times I$ or rank-one kernels with smooth factors. Such kernels define continuous bilinear forms on $H_0^s(I)$ for every $s>0$. The same conclusion therefore holds for $V_{\mathrm{log}}$ and $V_{\mathrm{cnt}}$.
\end{proof}

\subsection{The log field covariance}
\label{s:logfieldcovariance}
For $\phi\in L^\infty(I)$ define
\begin{equation}\label{eq:def-Fphi}
F_\phi(x)=\int_I \phi(E)\log|E-x|\,dE.
\end{equation}

The following lemma is classical (see, e.g., \cite[Exercises 1.4 q.4]{forrester2010log}), and we include the proof for completeness. 
\begin{lemma}\label{lem:logseries}
Let $\alpha\in(0,\pi)$. Then for almost every $\theta\in(-\pi,\pi)$,
\begin{equation}\label{eq:logseries}
\log|2\cos\alpha-2\cos\theta|
=
-2\sum_{k\ge1}\frac{\cos(k\alpha)\cos(k\theta)}{k},
\end{equation}
where the series converges in $L^2(-\pi,\pi)$ and locally uniformly away from $\theta=\pm\alpha$.
\end{lemma}

\begin{proof}
Recall the classical Fourier series
\[
-\log\big(2\big|\sin(t/2)\big|\big)=\sum_{k\ge1}\frac{\cos(kt)}{k},
\qquad t\in(0,2\pi),
\]
which holds in $L^2(0,2\pi)$ and pointwise away from multiples of $2\pi$. Since
\[
2\cos\alpha-2\cos\theta
=
-4\sin\frac{\alpha+\theta}{2}\sin\frac{\alpha-\theta}{2},
\]
we obtain
\begin{align*}
\log|2\cos\alpha-2\cos\theta|
&=
\log\left(2\left|\sin\frac{\alpha+\theta}{2}\right|\right)
+
\log\left(2\left|\sin\frac{\alpha-\theta}{2}\right|\right)
\\
&=
-\sum_{k\ge1}\frac{\cos(k(\alpha+\theta))}{k}
-\sum_{k\ge1}\frac{\cos(k(\alpha-\theta))}{k}
\\
&=
-2\sum_{k\ge1}\frac{\cos(k\alpha)\cos(k\theta)}{k}.
\qedhere
\end{align*}
\end{proof}

\begin{lemma}\label{lem:log-pointcoeff}
Let $E\in(-2,2)$ and choose $\alpha\in(0,\pi)$ such that $E=2\cos\alpha$. Define $F_E(x)=\log|E-x|$. Then, for every $k\ge1$,
\begin{equation}\label{eq:pointcoeff-log}
c_k(F_E)=-\frac{2}{k}\cos(k\alpha).
\end{equation}
\end{lemma}

\begin{proof}
Insert \eqref{eq:logseries} into \eqref{eq:ck-def} and use orthogonality of cosines:
\begin{align*}
c_k(F_E)
&=
\frac{1}{\pi}
\int_{-\pi}^{\pi}
\log|2\cos\alpha-2\cos\theta|\cos(k\theta)\,d\theta
\\
&=
\frac{1}{\pi}
\int_{-\pi}^{\pi}
\left(-2\sum_{m\ge1}\frac{\cos(m\alpha)\cos(m\theta)}{m}\right)
\cos(k\theta)\,d\theta
\\
&=
-\frac{2}{k}\cos(k\alpha).
\qedhere
\end{align*}
\end{proof}

\begin{lemma}\label{lem:log-testcoeff}
For $\phi\in C_c^\infty(I)$ and $k\ge1$,
\begin{equation}\label{eq:testcoeff-log}
c_k(F_\phi)
=
-\frac{2}{k}
\int_I \phi(E) T_k(E/2)\,dE,
\end{equation}
where $T_k$ is the $k$th Chebyshev polynomial of the first kind. Moreover, the sequence
\[
\int_I \phi(E)T_k(E/2)\,dE
\]
decays faster than any negative power of $k$.
\end{lemma}

\begin{proof}
Fubini's theorem and Lemma \ref{lem:log-pointcoeff} give
\[
c_k(F_\phi)
=
\int_I \phi(E)c_k(F_E)\,dE
=
-\frac{2}{k}
\int_I \phi(E)T_k(E/2)\,dE.
\]
For the decay, write $E=2\cos\alpha$ with $\alpha\in(0,\pi)$. Since $\phi$ is compactly supported in the spectral bulk, there exists $\widetilde\phi\in C_c^\infty((0,\pi))$ such that
\[
\int_I \phi(E)T_k(E/2)\,dE
=
\int_0^\pi \widetilde\phi(\alpha)\cos(k\alpha)\,d\alpha.
\]
Repeated integration by parts gives the claimed decay.
\end{proof}

\begin{proposition}\label{prop:log-covariance}
For all $\phi,\psi\in C_c^\infty(I)$, we have $
V(F_\phi,F_\psi)=V_{\mathrm{log}}(\phi,\psi)$.
\end{proposition}

\begin{proof}
By polarizing \eqref{eq:Q-def},
\begin{equation}\label{eq:polarized-log}
V(F_\phi,F_\psi)
=
\frac{1}{2}\sum_{k\ge1} k\, c_k(F_\phi)c_k(F_\psi)
+
\frac{s_4}{2} c_2(F_\phi)c_2(F_\psi).
\end{equation}
Using \eqref{eq:testcoeff-log},
\[
\frac{1}{2}\sum_{k\ge1} k\, c_k(F_\phi)c_k(F_\psi)
=
2\sum_{k\ge1}\frac{1}{k}
\left(\int_I \phi(E)T_k(E/2)\,dE\right)
\left(\int_I \psi(E')T_k(E'/2)\,dE'\right).
\]
The coefficient sequences decay rapidly, so Fubini's theorem yields
\[
\frac{1}{2}\sum_{k\ge1} k\, c_k(F_\phi)c_k(F_\psi)
=
\iint_{I\times I}
\phi(E)\psi(E')
\left(
2\sum_{k\ge1}\frac{T_k(E/2)T_k(E'/2)}{k}
\right)
dE\,dE'.
\]
Writing $E=2\cos\alpha$ and $E'=2\cos\beta$ and using \eqref{eq:logseries}, we get
\[
2\sum_{k\ge1}\frac{T_k(E/2)T_k(E'/2)}{k}
=
2\sum_{k\ge1}\frac{\cos(k\alpha)\cos(k\beta)}{k}
=
-\log|E-E'|.
\]

For the cumulant term, \eqref{eq:testcoeff-log} with $k=2$ gives
\[
c_2(F_\phi)
=
-\int_I \phi(E)T_2(E/2)\,dE
=
\frac{1}{2}\int_I \phi(E)(2-E^2)\,dE,
\]
because $T_2(t)=2t^2-1$. Substituting into \eqref{eq:polarized-log} yields the conclusion.
\end{proof}

\subsection{The counting field covariance}\label{s:countfieldcovariance}

For every $\phi \in L^\infty(I)$, with $\phi$ extended by zero outside $I$, set
\begin{equation}\label{eq:def-Gphi}
G_\phi(x)=\int_x^\infty \phi(E)\,dE.
\end{equation}

\begin{lemma}\label{lem:count-pointcoeff}
Let $E=2\cos\alpha$ with $\alpha\in(0,\pi)$ and set $H_E(x)=\1_{\{x\le E\}}$. Then, for every $k\ge1$,
\begin{equation}\label{eq:stepcoeff}
c_k(H_E)=-\frac{2}{\pi k}\sin(k\alpha).
\end{equation}
\end{lemma}

\begin{proof}
By definition,
\[
c_k(H_E)
=
\frac{1}{\pi}\int_{-\pi}^{\pi}
\1_{\{2\cos\theta\le E\}}\cos(k\theta)\,d\theta.
\]
Since $2\cos\theta\le2\cos\alpha$ if and only if $|\theta|\ge\alpha$, we obtain
\[
c_k(H_E)
=
\frac{2}{\pi}\int_\alpha^\pi \cos(k\theta)\,d\theta
=
-\frac{2}{\pi k}\sin(k\alpha).
\]
\end{proof}

\begin{lemma}\label{lem:count-testcoeff}
For all $\phi\in C_c^\infty(I)$ and every $k\ge1$,
\begin{align}
c_k(G_\phi)
&=
-\frac{2}{\pi k}\int_I \phi(E)\sin\big(k\alpha(E)\big)\,dE,
\label{eq:countcoeff1}
\\
&=
-\frac{1}{\pi k}\int_I \phi(E)U_{k-1}(E/2)\sqrt{4-E^2}\,dE,
\label{eq:countcoeff2}
\end{align}
where $\alpha(E)=\arccos(E/2)$ and $U_{k-1}$ is the $(k-1)$st Chebyshev polynomial of the second kind. Further, the sequence $(c_k(G_\phi))_{k\ge1}$ decays faster than any negative power of $k$.
\end{lemma}

\begin{proof}
Since $G_\phi(x)=\int_I \phi(E)H_E(x)\,dE$, Fubini's theorem and Lemma \ref{lem:count-pointcoeff} give
\[
c_k(G_\phi)
=
\int_I \phi(E)c_k(H_E)\,dE
=
-\frac{2}{\pi k}
\int_I \phi(E)\sin(k\alpha(E))\,dE,
\]
which is \eqref{eq:countcoeff1}. Using
\[
\sin(k\alpha)=U_{k-1}(\cos\alpha)\sin\alpha
=
U_{k-1}(E/2)\frac{\sqrt{4-E^2}}{2}
\]
yields \eqref{eq:countcoeff2}.

For the decay, write $E=2\cos\theta$ with $\theta\in(0,\pi)$. Then
\[
\int_I \phi(E)\sin(k\alpha(E))\,dE
=
2\int_0^\pi \widetilde\phi(\theta)\sin(k\theta)\,d\theta
\]
for some $\widetilde\phi\in C_c^\infty((0,\pi))$. Repeated integration by parts finishes the proof.
\end{proof}

\begin{lemma}\label{lem:sineseries}
For all $\alpha,\beta\in(0,\pi)$ such that $\alpha \neq \beta$,
\begin{equation}\label{eq:sineseries}
2\sum_{k=1}^\infty \frac{\sin(k\alpha)\sin(k\beta)}{k}
=
\log\left|
\frac{\sin\frac{\alpha+\beta}{2}}{\sin\frac{\alpha-\beta}{2}}
\right|.
\end{equation}
\end{lemma}

\begin{proof}
The conclusion follows from
\[
2\sin(k\alpha)\sin(k\beta)=\cos(k(\alpha-\beta))-\cos(k(\alpha+\beta))
\]
and the Fourier series
\[
\sum_{k=1}^\infty \frac{\cos(kt)}{k}
=
-\log\big(2\big|\sin(t/2)\big|\big),
\qquad t\in(0,2\pi).
\]
\end{proof}

\begin{proposition}\label{prop:count-covariance}
For $\phi,\psi\in C_c^\infty(I)$, we have $
V(G_\phi,G_\psi)=V_{\mathrm{cnt}}(\phi,\psi)$. 
\end{proposition}

\begin{proof}
By polarizing \eqref{eq:Q-def},
\begin{equation}\label{eq:polarized-count}
V(G_\phi,G_\psi)
=
\frac{1}{2}\sum_{k\ge1} k\, c_k(G_\phi)c_k(G_\psi)
+
\frac{s_4}{2}c_2(G_\phi)c_2(G_\psi).
\end{equation}
Using \eqref{eq:countcoeff1},
\[
\frac{1}{2}\sum_{k\ge1} k\, c_k(G_\phi)c_k(G_\psi)
=
\frac{2}{\pi^2}
\sum_{k\ge1}\frac{1}{k}
\left(\int_I \phi(E)\sin(k\alpha(E))\,dE\right)
\left(\int_I \psi(E')\sin(k\alpha(E'))\,dE'\right).
\]
Rapid decay of the sequence $(c_k(G_\phi))_{k\ge1}$ (from Lemma~\ref{lem:count-testcoeff}) justifies the use of Fubini's theorem, so Lemma \ref{lem:sineseries} gives
\[
\frac{1}{2}\sum_{k\ge1} k\, c_k(G_\phi)c_k(G_\psi)
=
\iint_{I\times I}
\phi(E)\psi(E')
\frac{1}{\pi^2}
\log\left|
\frac{\sin\frac{\alpha(E)+\alpha(E')}{2}}
{\sin\frac{\alpha(E)-\alpha(E')}{2}}
\right|
dE\,dE'.
\]
Using
\[
E-E'=2\cos\alpha(E)-2\cos\alpha(E')
=
-4\sin\frac{\alpha(E)+\alpha(E')}{2}
\sin\frac{\alpha(E)-\alpha(E')}{2},
\]
we rewrite the logarithmic  kernel as
\[
-\frac{1}{\pi^2}\log|E-E'|
+
\frac{1}{\pi^2}
\log\left(
\frac{4-EE'+\sqrt{4-E^2}\sqrt{4-E'^2}}{2}
\right).
\]
For the cumulant term, \eqref{eq:countcoeff2} with $k=2$ yields
\[
c_2(G_\phi)
=
-\frac{1}{2\pi}
\int_I E\sqrt{4-E^2}\,\phi(E)\,dE.
\]
Substituting this into \eqref{eq:polarized-count} produces exactly the rank-one term in \eqref{eq:Kcnt}.
\end{proof}

\section{The Log-Determinant Field}
\label{s:logfieldproof}

This section is devoted to the proof of Theorem~\ref{thm:log-main}.
For a given $s>0$ we fix
an auxiliary exponent
\[
    0<\sigma<\min\{s,1/2\}.
\]
For such $\sigma$, $C_c^\infty(I)$ is dense in $H_0^\sigma(I)$ (see, e.g., \cite[(3.5a)]{ern2021finite}). We fix a
countable set
\[
    \mathcal D_\sigma\subset C_c^\infty(I)
\]
that is dense in $H_0^\sigma(I)$.
\subsection{Preliminary estimates}
For all $\phi\in L^\infty(I)$, Fubini's theorem gives
\begin{equation}\label{eq:log-linear-stat}
\ip{X_N^{\mathrm{log}}}{\phi}
=
\mathcal N_N(F_\phi).
\end{equation}
We recall the choices of $\eta$ and $\chi$ from \eqref{eq:cutoffs}. 
\begin{lemma}\label{lem:tail-negligible-log}
For every $\phi\in C_c^\infty(I)$, we have the $L^2$-limit
\[
\lim_{N \rightarrow \infty} \mathcal N_N\big((1-\eta)F_\phi\big)= 0.
\]
\end{lemma}

\begin{proof}
The function $(1-\eta)F_\phi$ vanishes on $[-3,3]$, and 
\[
|F_\phi(x)|\le C_\phi \log(2+|x|),\qquad x\in\R.
\]
Hence
\[
\left|\Tr (1-\eta)F_\phi(H_N)\right|
\le
N\,C_\phi \log(2+\|H_N\|)\,\1_{\{\|H_N\|>3\}}.
\]
Moreover,
\[
\left\|\mathcal N_N\bigl((1-\eta)F_\phi\bigr)\right\|_2
\le
2\left\|\Tr (1-\eta)F_\phi(H_N)\right\|_2 .
\]
Therefore
\begin{align*}
\E\left|\mathcal N_N\bigl((1-\eta)F_\phi\bigr)\right|^2
&\le
C_\phi N^2 \E\!\left[\log^2(2+\|H_N\|)\1_{\{\|H_N\|>3\}}\right]
\\
&\le
C_\phi N^2 \Bigl(\E \log^4(2+\|H_N\|)\Bigr)^{1/2}
\Pbb(\|H_N\|>3)^{1/2}
\\
&\le
C_\phi N^2 e^{-c(\log N)^2},
\end{align*}
where we used the Cauchy--Schwarz inequality and Proposition~\ref{prop:norm-tail}. In particular, $\mathcal N_N((1-\eta)F_\phi)\to0$ in $L^2$.
\end{proof}

Given a distribution $f$ on $I$, define
\[
\mathcal F f(x)=\int_I f(E)\log|x-E|\,dE
\]
whenever the right-hand side is well defined. We will use
\begin{equation}\label{eq:log-decomp}
\eta\,\mathcal F f = Bf+Rf,
\qquad
Bf=\chi\,\mathcal F f,
\qquad
Rf=(\eta-\chi)\mathcal F f.
\end{equation}

\begin{lemma}\label{lem:log-remainder-smoothing}
For every $q\ge0$, the operator $R$ from \eqref{eq:log-decomp} 
extends continuously from $L^2(I)$ to $H^q(\mathbb R)$. In particular, if
$\phi\in C_c^\infty(I)$, then $R\phi\in C_c^\infty(\mathbb R)$. 
Moreover, for every $q\ge0$ there exists $C_q>0$ such that
\[
\|Re_k\|_{H^q(\mathbb R)}
\le C_q(1+\mu_k)^{-1/2},
\qquad k\ge1.
\]
\end{lemma}

\begin{proof}
Let
\[
K_R(x,E)=\big(\eta(x)-\chi(x)\big)\log|x-E|.
\]
Since $\chi\equiv 1$ on a neighborhood of $I$ and $\eta\equiv1$ on
$[-3,3]$, and $I\subset [-3,3]$, the support of $\eta-\chi$ is compact and disjoint
from $I$. Hence $
K_R\in C_c^\infty(\mathbb R\times I)$. 

We first prove the $L^2(I)\to H^q(\mathbb R)$ estimate. It suffices to
consider integers $q\ge0$, since the general case follows by the
continuous embedding $H^n(\mathbb R)\hookrightarrow H^q(\mathbb R)$ for
any integer $n\ge q$. For $0\le j\le q$,
\[
\partial_x^j Rf(x)
=
\int_I \partial_x^j K_R(x,E)f(E)\,dE.
\]
By the Cauchy--Schwarz inequality in $E$,
\[
|\partial_x^j Rf(x)|^2
\le
\|f\|_{L^2(I)}^2
\int_I |\partial_x^jK_R(x,E)|^2\,dE.
\]
Integrating in $x$ gives
\[
\|\partial_x^j Rf\|_{L^2(\mathbb R)}
\le
C_j\|f\|_{L^2(I)}.
\]
Summing over $0\le j\le q$, we obtain
\[
\|Rf\|_{H^q(\mathbb R)}
\le
C_q\|f\|_{L^2(I)}.
\]
In particular, if $\phi\in C_c^\infty(I)$, then $R\phi\in C_c^\infty(\mathbb R)$.

It remains to prove the basis estimate. Set $\kappa_k = \pi k /L$. 
Then
\[
Re_k(x)
=
\sqrt{\frac2L}\int_a^b K_R(x,E)\sin(\kappa_k(E-a))\,dE.
\]
Integrating by parts once in $E$ gives
\[
Re_k(x)
=
\sqrt{\frac{2}{L}}\frac1{\kappa_k}
\left[
K_R(x,a)-(-1)^kK_R(x,b)
+
\int_a^b \partial_EK_R(x,E)\cos(\kappa_k(E-a))\,dE
\right].
\]
Since $K_R$ and $\partial_EK_R$ are smooth in $x$ and compactly supported
in $x$, uniformly for $E\in I$, the same estimate after applying
$\partial_x^j$, for all $0\le j\le q$, gives
\[
\|Re_k\|_{H^q(\mathbb R)}
\le
\frac{C_q}{\kappa_k}.
\]
Because $\mu_k=\kappa_k^2$, this is
\[
\|Re_k\|_{H^q(\mathbb R)}
\le
C_q(1+\mu_k)^{-1/2},
\qquad k\ge1.
\]
This completes the proof.
\end{proof}

\begin{lemma}\label{lem:log-fdd}
For every $\phi\in C_c^\infty(I)$,
$\ip{X_N^{\mathrm{log}}}{\phi}$
converges in law to a mean-zero Gaussian random variable with variance
$V_{\mathrm{log}}(\phi,\phi)$.
\end{lemma}

\begin{proof}
By Lemma \ref{lem:tail-negligible-log},
\[
\ip{X_N^{\mathrm{log}}}{\phi}
=
\mathcal N_N(\eta F_{\phi}) + o_{L^2}(1).
\]
Since $\phi\in C_c^\infty(I)$, its extension by zero belongs to
$C_c^\infty(\mathbb R)$, and
\[
F_\phi=(\log|\cdot|)*\phi
\]
in the sense of distributions. Hence $F_\phi\in C^\infty(\mathbb R)$.
Since $\eta\in C_c^\infty(\mathbb R)$, it follows that
$\eta F_\phi\in C_c^\infty(\mathbb R)$. Proposition
\ref{prop:smooth-clt} therefore gives
\[
\mathcal N_N(\eta F_\phi)
\Rightarrow
\mathcal N(0,V(\eta F_\phi,\eta F_\phi)).
\]
The Chebyshev coefficients in \eqref{eq:ck-def} depend only on the
restriction of the test function to $[-2,2]$, and $\eta\equiv1$ on
$[-2,2]$. Therefore
\[
V(\eta F_\phi,\eta F_\phi)=V(F_\phi,F_\phi).
\]
By Proposition \ref{prop:log-covariance},
\[
V(F_\phi,F_\phi)=V_{\mathrm{log}}(\phi,\phi).
\]
Combining this with the $o_{L^2}(1)$ error proves the claim.
\end{proof}

\subsection{Uniform Sobolev bounds}

\begin{lemma}\label{lem:log-tail-basis}
For every $k\ge1$, let
\[
T_k^{\mathrm{log}}=(1-\eta)F_{e_k}.
\]
Then there exists $C>0$ such that
\begin{equation}\label{eq:log-tail-basis-var}
\sup_{N\ge1}\Var\bigl(\mathcal N_N(T_k^{\mathrm{log}})\bigr)\le C k^{-2}.
\end{equation}
\end{lemma}

\begin{proof}

Since $\eta\equiv1$ on $[-3,3]$, it suffices to estimate $F_{e_k}(x)$ for
$|x|>3$. For such $x$ and $E\in I\Subset(-2,2)$, the function
$E\mapsto \log|x-E|$ is smooth and
\[
\big |\log|x-a||+|\log|x-b|\big|\le C\log(2+|x|),
\qquad
\int_a^b \frac{dE}{|x-E|}\le C.
\]
Writing $\kappa_k=\pi k/L$, integration by parts gives
\[
F_{e_k}(x)
=
\sqrt{\frac2L}\frac1{\kappa_k}
\left[
\log|x-a|-(-1)^k\log|x-b|
+
\int_a^b \frac{\cos(\kappa_k(E-a))}{E-x}\,dE
\right].
\]
Therefore
\[
|F_{e_k}(x)|
\le Ck^{-1}\log(2+|x|),
\qquad |x|>3.
\]
This implies
\[
|\Tr T_k^{\mathrm{log}}(H_N)|
\le
C N k^{-1}\log(2+\|H_N\|)\,\1_{\{\|H_N\|>3\}}.
\]
Moreover,
\[
\|\mathcal N_N(T_k^{\mathrm{log}})\|_2
\le
2\|\Tr T_k^{\mathrm{log}}(H_N)\|_2 .
\]
Hence
\begin{align*}
\Var\bigl(\mathcal N_N(T_k^{\mathrm{log}})\bigr)
&\le
\E |\mathcal N_N(T_k^{\mathrm{log}})|^2
\\
&\le
C k^{-2} N^2 \E\!\left[\log^2(2+\|H_N\|)\1_{\{\|H_N\|>3\}}\right]
\\
&\le
C k^{-2} N^2 \Bigl(\E \log^4(2+\|H_N\|)\Bigr)^{1/2}
\Pbb(\|H_N\|>3)^{1/2}
\\
&\le
C k^{-2},
\end{align*}
for all $N\ge1$, by Proposition \ref{prop:norm-tail}. This proves \eqref{eq:log-tail-basis-var}.
\end{proof}

\begin{lemma}\label{lem:log-coeff-variance}
Fix $0<\varepsilon<1/2$. There exists $C_{\varepsilon}>0$ such that for all
$N,k\ge1$,
\[
\Var\bigl(\ip{X_N^{\mathrm{log}}}{e_k}\bigr)
\le
C_{\varepsilon}(1+\mu_k)^{-1/2+\varepsilon}
+
C_{\varepsilon}k^{-2}.
\]
\end{lemma}

\begin{proof}
By \eqref{eq:log-linear-stat} and the decomposition
\[
F_{e_k}=Be_k+Re_k+T_k^{\mathrm{log}},
\]
we have
\[
\ip{X_N^{\mathrm{log}}}{e_k}
=
\mathcal N_N(Be_k)+\mathcal N_N(Re_k)+\mathcal N_N(T_k^{\mathrm{log}}).
\]
Hence
\[
\Var\bigl(\ip{X_N^{\mathrm{log}}}{e_k}\bigr)
\le
3\Var(\mathcal N_N(Be_k))
+3\Var(\mathcal N_N(Re_k))
+3\Var(\mathcal N_N(T_k^{\mathrm{log}})).
\]

Since $Be_k$ is supported in $J\Subset(-2,2)$, Theorem \ref{thm:LS} and Lemma \ref{lem:log-operator} give
\[
\Var(\mathcal N_N(Be_k))
\le
C_\varepsilon \|Be_k\|_{H^{1/2+\varepsilon}(\R)}^2
\le
C_\varepsilon \|e_k\|_{H^{-1/2+\varepsilon}(I)}^2
=
C_\varepsilon (1+\mu_k)^{-1/2+\varepsilon}.
\]
For the smooth remainder, Theorem \ref{thm:SW} and Lemma
\ref{lem:log-remainder-smoothing} yield
\[
\Var(\mathcal N_N(Re_k))
\le
C_{\varepsilon}\|Re_k\|_{H^{1+\varepsilon}(\R)}^2
\le
C_{\varepsilon}(1+\mu_k)^{-1}
\le
C_{\varepsilon}k^{-2}.
\]
The tail term is bounded by Lemma \ref{lem:log-tail-basis}. Combining the three
estimates gives
\[
\Var\bigl(\ip{X_N^{\mathrm{log}}}{e_k}\bigr)
\le
C_{\varepsilon}(1+\mu_k)^{-1/2+\varepsilon}
+
C_{\varepsilon}k^{-2}.
\]
This implies the conclusion after increasing the constant
$C_{\varepsilon}$.
\end{proof}

\begin{lemma}\label{lem:log-Hr}
For every $r>0$,
\[
\sup_{N\ge1}\E \|X_N^{\mathrm{log}}\|_{H^{-r}(I)}^2 < \infty.
\]
\end{lemma}

\begin{proof}
By Tonelli's theorem and the definition of the norm for $H^{-r}(I)$,
\[
\E \|X_N^{\mathrm{log}}\|_{H^{-r}(I)}^2
=
\sum_{k\ge1}(1+\mu_k)^{-r}\Var\bigl(\ip{X_N^{\mathrm{log}}}{e_k}\bigr).
\]
Apply Lemma \ref{lem:log-coeff-variance} with some
$\varepsilon\in(0,\min\{r,1/2\})$. Since $\mu_k\asymp k^2$,
\[
\sum_{k\ge1}(1+\mu_k)^{-r-1/2+\varepsilon}
\asymp
\sum_{k\ge1}k^{-2r-1+2\varepsilon}<\infty,
\]
while the remaining series
\[
\sum_{k\ge1}(1+\mu_k)^{-r}k^{-2}
\]
also converges.
\end{proof}

\subsection{Proof of Theorem~\ref{thm:log-main}}

\begin{proof}[Proof of Theorem~\ref{thm:log-main}]
Fix $s>0$ and recall that $\sigma$ and $\mathcal D_\sigma$ were fixed at the beginning of this section.

Choose $r\in(0,\sigma)$. By Lemma \ref{lem:log-Hr} and Proposition
\ref{prop:tightness-from-stronger}, the sequence $(X_N^{\mathrm{log}})$
is tight in $H^{-\sigma}(I)$. By Lemma~\ref{lem:log-fdd}, for every
$\phi\in\mathcal D_\sigma$, $\ip{X_N^{\mathrm{log}}}{\phi}$ converges in
law to a centered Gaussian random variable with variance
$V_{\mathrm{log}}(\phi,\phi)$. Proposition~\ref{prop:convergence-criterion}
therefore yields convergence in law in $H^{-\sigma}(I)$. We denote the limit
by $X_\sigma^{\mathrm{log}}$.

We next identify this limit. Let $\phi\in C_c^\infty(I)$. Since the map
$u\mapsto \ip{u}{\phi}$ is continuous on $H^{-\sigma}(I)$, the convergence
$X_N^{\mathrm{log}}\Rightarrow X_\sigma^{\mathrm{log}}$ implies $\ip{X_N^{\mathrm{log}}}{\phi}
    \Rightarrow
    \ip{X_\sigma^{\mathrm{log}}}{\phi}$. 
On the other hand, Lemma~\ref{lem:log-fdd} gives
\[
    \ip{X_N^{\mathrm{log}}}{\phi}
    \Rightarrow
    \mathcal N(0,V_{\mathrm{log}}(\phi,\phi)).
\]
Hence
\[
    \E e^{i\ip{X_\sigma^{\mathrm{log}}}{\phi}}
    =
    \exp\!\left(-\frac12 V_{\mathrm{log}}(\phi,\phi)\right),
    \qquad \phi\in C_c^\infty(I).
\]
By the continuity of characteristic functionals and of
$V_{\mathrm{log}}$ on $H_0^\sigma(I)$, and by the density of
$C_c^\infty(I)$ in $H_0^\sigma(I)$, the same identity holds for every
$h\in H_0^\sigma(I)$. Thus $X_\sigma^{\mathrm{log}}$ is centered Gaussian,
and its covariance form is $V_{\mathrm{log}}$ by polarization.

Since $\sigma<s$, the embedding
\[
    \iota_{\sigma,s}:H^{-\sigma}(I)\hookrightarrow H^{-s}(I)
\]
is continuous by Proposition~\ref{prop:compact-embedding}. Hence, by the
continuous mapping theorem,
\[
    X_N^{\mathrm{log}}
    =
    \iota_{\sigma,s}X_N^{\mathrm{log}}
    \Rightarrow
    \iota_{\sigma,s}X_\sigma^{\mathrm{log}}
\]
in law in $H^{-s}(I)$. We define $
    X^{\mathrm{log}} =\iota_{\sigma,s}X_\sigma^{\mathrm{log}}$.
Since $X_\sigma^{\mathrm{log}}$ is centered Gaussian with covariance form
$V_{\mathrm{log}}$ on $H_0^\sigma(I)$, its image
$X^{\mathrm{log}}=\iota_{\sigma,s}X_\sigma^{\mathrm{log}}$ is centered
Gaussian in $H^{-s}(I)$, and for $h,g\in H_0^s(I)\subset H_0^\sigma(I)$,
\[
    \E\bigl[\ip{X^{\mathrm{log}}}{h}\ip{X^{\mathrm{log}}}{g}\bigr]
    =
    V_{\mathrm{log}}(h,g),
\]
as claimed.
\end{proof}

\section{The Counting Field}\label{s:countingfieldproof}

The structure of this section parallels that of Section~\ref{s:logfieldproof} and  concludes with a proof of Theorem~\ref{thm:count-main}. 
We retain the notation $\mathcal D_\sigma$ for the countable set fixed at the beginning of the previous section. 
\subsection{Preliminary estimates}

For $\phi\in L^\infty(I)$, Fubini's theorem applied to the step functions $x\mapsto \1_{\{x\le E\}}$ gives
\begin{equation}\label{eq:count-linear-stat}
\ip{X_N^{\mathrm{cnt}}}{\phi}
=
\mathcal N_N(G_\phi).
\end{equation}

\begin{lemma}\label{lem:tail-negligible-count}
For every $\phi\in C_c^\infty(I)$, we have the $L^2$-limit 
\[
\lim_{N \rightarrow \infty} 
\mathcal N_N\big((1-\eta)G_\phi\big) = 0.
\]
\end{lemma}

\begin{proof}
The function $(1-\eta)G_\phi$ vanishes on $[-3,3]$ and is bounded. Hence
\[
\left|\Tr (1-\eta)G_\phi(H_N)\right|
\le
N\|(1-\eta)G_\phi\|_\infty\,\1_{\{\|H_N\|>3\}}.
\]
Moreover,
\[
\left\|\mathcal N_N\bigl((1-\eta)G_\phi\bigr)\right\|_2
\le
2\left\|\Tr (1-\eta)G_\phi(H_N)\right\|_2 .
\]
Therefore
\[
\E\left|\mathcal N_N\bigl((1-\eta)G_\phi\bigr)\right|^2
\le
C_\phi N^2 \Pbb(\|H_N\|>3)
\le
C_\phi N^2 e^{-c(\log N)^2}
\]
by Proposition \ref{prop:norm-tail}. In particular, $\mathcal N_N((1-\eta)G_\phi)\to0$ in $L^2$.
\end{proof}

\begin{lemma}\label{lem:count-fdd}
For every $\phi \in C_c^\infty(I)$,
$\ip{X_N^{\mathrm{cnt}}}{\phi}$
converges in law to a mean-zero Gaussian with variance $V_{\mathrm{cnt}}(\phi,\phi)$.
\end{lemma}

\begin{proof}
By Lemma \ref{lem:tail-negligible-count},
\[
\ip{X_N^{\mathrm{cnt}}}{\phi}
=
\mathcal N_N(\eta G_{\phi})+o_{L^2}(1).
\]
Since $\eta G_{\phi}\in C_c^\infty(\R)$, Proposition \ref{prop:smooth-clt} implies that $\mathcal N_N(\eta G_{\phi})$ 
converges to a centered Gaussian with variance $V(\eta G_{\phi},\eta G_{\phi})$. 
Because $\eta\equiv1$ on $[-2,2]$, the Chebyshev coefficients are unchanged by inserting $\eta$, so
\[
V(\eta G_{\phi},\eta G_{\phi})=V(G_{\phi},G_{\phi}).
\]
Proposition \ref{prop:count-covariance} identifies this with $V_{\mathrm{cnt}}(\phi,\phi)$.
\end{proof}

\subsection{Uniform Sobolev bounds}

For $f\in L^2(I)$ we write
\[
G_f(x)=\int_x^\infty f(E)\,dE.
\]
We decompose
\[
\eta G_f = \chi G_f + (\eta-\chi)G_f.
\]

\begin{lemma}\label{lem:count-bulk-norms}
For every $k\ge1$,
\begin{equation}\label{eq:count-L2}
\|\chi G_{e_k}\|_{L^2(\R)}\le Ck^{-1}
\end{equation}
and
\begin{equation}\label{eq:count-H1}
\|\chi G_{e_k}\|_{H^1(\R)}\le C.
\end{equation}
Consequently, for every $0<\varepsilon<1/2$,
\begin{equation}\label{eq:count-Hhalf}
\|\chi G_{e_k}\|_{H^{1/2+\varepsilon}(\R)}^2\le C_\varepsilon k^{-1+2\varepsilon}.
\end{equation}
\end{lemma}

\begin{proof}
For $x\in[a,b]$,
\[
G_{e_k}(x)=\int_x^b e_k(E)\,dE
=
\sqrt{\frac{2}{L}}\frac{L}{\pi k}
\left(
\cos\left(\frac{\pi k(x-a)}{L}\right)-(-1)^k
\right),
\]
while for $x<a$,
\[
G_{e_k}(x)=\int_a^b e_k(E)\,dE=\cO(k^{-1}),
\]
and for $x>b$ we have $G_{e_k}(x)=0$. Since $\chi$ is compactly supported, this immediately gives \eqref{eq:count-L2}.

For the $H^1$ bound, note that
\[
\partial_x(\chi G_{e_k})
=
\chi' G_{e_k}-\chi e_k.
\]
The first term is $\cO(k^{-1})$ in $L^2$ by \eqref{eq:count-L2}, and the second is bounded uniformly in $L^2$ because $\|e_k\|_{L^2(I)}=1$. This proves \eqref{eq:count-H1}. Finally, \eqref{eq:count-Hhalf} follows from the well-known interpolation inequality
\[
\|u\|_{H^\theta(\R)}\le C_\theta \|u\|_{L^2(\R)}^{1-\theta}\|u\|_{H^1(\R)}^\theta,
\qquad 0\le\theta\le1,
\]
with the choice $\theta=1/2+\varepsilon$ (which is immediate from H\"{o}lder's inequality).
\end{proof}

\begin{lemma}\label{lem:count-remainder-smoothing}
For every $q\ge0$, the operator
\[
f\mapsto (\eta-\chi)G_f
\]
sends $L^2(I)$ continuously into $H^q(\mathbb R)$. Moreover, for every
$q\ge0$ there exists $C_q>0$ such that
\begin{equation}\label{eq:count-remainder-basis}
\|(\eta-\chi)G_{e_k}\|_{H^q(\mathbb R)}
\le C_q(1+\mu_k)^{-1/2},
\qquad k\ge1.
\end{equation}
\end{lemma}

\begin{proof}
The kernel
\[
K(x,E)=(\eta(x)-\chi(x))\mathbf 1_{\{x\le E\}}
\]
is smooth on $\mathbb R\times I$. Indeed, $\operatorname{supp}(\eta-\chi)$
is compact and disjoint from $I$. On the part of
$\operatorname{supp}(\eta-\chi)$ lying to the left of $I$, one has
$\mathbf 1_{\{x\le E\}}=1$ for all $E\in I$, while on the part lying to the
right of $I$, one has $\mathbf 1_{\{x\le E\}}=0$ for all $E\in I$.

Thus $K$ is smooth, compactly supported in $x$, and all its $x$-derivatives
are bounded uniformly for $E\in I$. Therefore, by the Cauchy--Schwarz inequality in $E$
and differentiation under the integral, the associated operator maps
$L^2(I)$ continuously into $H^q(\mathbb R)$ for every $q\ge0$.

For the basis estimate, observe that
\[
G_{e_k}(x)=\int_I e_k(E)\,dE
\]
on the part of $\operatorname{supp}(\eta-\chi)$ lying to the left of $I$,
while $G_{e_k}(x)=0$ on the part lying to the right of $I$. Hence
\[
(\eta-\chi)G_{e_k}
=
\left(\int_I e_k(E)\,dE\right)(\eta-\chi)\mathbf 1_{\{x<a\}}.
\]
The function $(\eta-\chi)\mathbf 1_{\{x<a\}}$ is smooth and compactly
supported, because $\eta-\chi$ vanishes on a neighborhood of $I$. Therefore
\[
\|(\eta-\chi)G_{e_k}\|_{H^q(\mathbb R)}
\le
C_q\left|\int_I e_k(E)\,dE\right|.
\]
Since
\[
\int_I e_k(E)\,dE
=
\sqrt{\frac{2}{L}}\frac{L}{\pi k}\bigl(1-(-1)^k\bigr),
\]
we get
\[
\|(\eta-\chi)G_{e_k}\|_{H^q(\mathbb R)}
\le C_q k^{-1}
\le C_q(1+\mu_k)^{-1/2}.
\]
\end{proof}

\begin{lemma}\label{lem:count-tail-basis}
For every $k\ge1$, let
\[
T_k^{\mathrm{cnt}}=(1-\eta)G_{e_k}.
\]
Then
\begin{equation}\label{eq:count-tail-var}
\sup_{N\ge1}\Var\bigl(\mathcal N_N(T_k^{\mathrm{cnt}})\bigr)\le Ck^{-2}.
\end{equation}
\end{lemma}

\begin{proof}
Since $G_{e_k}$ is zero to the right of $I$ and equals $\int_I e_k(E)\,dE=\cO(k^{-1})$ to the left of $I$, we have
\[
\|T_k^{\mathrm{cnt}}\|_\infty\le Ck^{-1}.
\]
Therefore
\[
|\Tr T_k^{\mathrm{cnt}}(H_N)|
\le
CNk^{-1}\1_{\{\|H_N\|>3\}}.
\]
Moreover,
\[
\|\mathcal N_N(T_k^{\mathrm{cnt}})\|_2
\le
2\|\Tr T_k^{\mathrm{cnt}}(H_N)\|_2 .
\]
Hence
\[
\Var\bigl(\mathcal N_N(T_k^{\mathrm{cnt}})\bigr)
\le
\E |\mathcal N_N(T_k^{\mathrm{cnt}})|^2
\le
C k^{-2} N^2 \Pbb(\|H_N\|>3)
\le
C k^{-2}
\]
by Proposition \ref{prop:norm-tail}. This proves \eqref{eq:count-tail-var}.
\end{proof}

\begin{lemma}\label{lem:count-coeff-variance}
Fix $0<\varepsilon<1/2$. There exists $C_{\varepsilon}>0$ such that for all
$N,k\ge1$,
\[
\Var\bigl(\ip{X_N^{\mathrm{cnt}}}{e_k}\bigr)
\le
C_{\varepsilon} k^{-1+2\varepsilon}
+
C_{\varepsilon}k^{-2}.
\]
\end{lemma}

\begin{proof}
By \eqref{eq:count-linear-stat},
\[
\ip{X_N^{\mathrm{cnt}}}{e_k}
=
\mathcal N_N(\chi G_{e_k})
+
\mathcal N_N((\eta-\chi)G_{e_k})
+
\mathcal N_N(T_k^{\mathrm{cnt}}).
\]
Thus
\begin{align*}
\Var\bigl(\ip{X_N^{\mathrm{cnt}}}{e_k}\bigr)
&\le
3\Var(\mathcal N_N(\chi G_{e_k}))
+3\Var(\mathcal N_N((\eta-\chi)G_{e_k}))
+3\Var(\mathcal N_N(T_k^{\mathrm{cnt}})).
\end{align*}
Theorem \ref{thm:LS} and \eqref{eq:count-Hhalf} give
\[
\Var(\mathcal N_N(\chi G_{e_k}))
\le
C_\varepsilon \|\chi G_{e_k}\|_{H^{1/2+\varepsilon}(\R)}^2
\le
C_\varepsilon k^{-1+2\varepsilon}.
\]
Theorem \ref{thm:SW} and Lemma \ref{lem:count-remainder-smoothing} imply
\[
\Var(\mathcal N_N((\eta-\chi)G_{e_k}))
\le
C_{\varepsilon}\|(\eta-\chi)G_{e_k}\|_{H^{1+\varepsilon}(\mathbb R)}^2
\le
C_{\varepsilon}(1+\mu_k)^{-1}
\le
C_{\varepsilon}k^{-2}.
\]
The tail term is handled by Lemma \ref{lem:count-tail-basis}. Combining the
three estimates gives
\[
\Var\bigl(\ip{X_N^{\mathrm{cnt}}}{e_k}\bigr)
\le
C_{\varepsilon}k^{-1+2\varepsilon}
+
C_{\varepsilon}k^{-2}.
\]
This implies the conclusion after increasing the constant
$C_{\varepsilon}$.
\end{proof}

\begin{lemma}\label{lem:count-Hr}
For every $r>0$,
\[
\sup_{N\ge1}\E \|X_N^{\mathrm{cnt}}\|_{H^{-r}(I)}^2 < \infty.
\]
\end{lemma}

\begin{proof}
As before,
\[
\E \|X_N^{\mathrm{cnt}}\|_{H^{-r}(I)}^2
=
\sum_{k\ge1}(1+\mu_k)^{-r}\Var\bigl(\ip{X_N^{\mathrm{cnt}}}{e_k}\bigr).
\]
Apply Lemma \ref{lem:count-coeff-variance} with some $\varepsilon\in(0,\min\{r,1/2\})$. Since $\mu_k\asymp k^2$,
\[
\sum_{k\ge1}(1+\mu_k)^{-r}k^{-1+2\varepsilon}
\asymp
\sum_{k\ge1}k^{-2r-1+2\varepsilon}<\infty.
\]
The two remaining series are summable, so the uniform bound follows.
\end{proof}

\subsection{Proof of Theorem \ref{thm:count-main}}

\begin{proof}[Proof of Theorem \ref{thm:count-main}]
Given the preceding lemmas, the proof is nearly identical to that of Theorem~\ref{thm:log-main}, so we omit the details. 
\end{proof}

\printbibliography

\end{document}